\documentclass{amsart}
\usepackage{subfigure}

\usepackage{graphicx,rotating}
\usepackage[left=1.5in,top=1.2in,right=1.5in,bottom=1.2in,head=.2in]{geometry}
\newtheorem{proposition}{Proposition}[section]
\newtheorem{theorem}[proposition]{Theorem}

\theoremstyle{definition}

\newtheorem{remark}[proposition]{Remark}

\newcommand{\tb}{\text{tb}}
\newcommand{\rot}{\text{rot}}

\newcommand{\lk}{\text{lk}}
\newcommand{\Wh}{\text{Wh}}
\begin{document}

%

\title{A new family of links topologically, but not smoothly,\\concordant to the Hopf link}

\author{Christopher W. Davis}
\address{Department of Mathematics, The University of Wisconsin at Eau Claire}
\email{daviscw@uwec.edu}
\urladdr{http://people.uwec.edu/daviscw/}

\author{Arunima Ray}
\address{Department of Mathematics, Brandeis University}
\email{aruray@brandeis.edu}
\urladdr{http://people.brandeis.edu/~aruray/}

\date{\today}

\begin{abstract}
We give new examples of 2--component links with linking number one and unknotted components that are topologically concordant to the positive Hopf link, but not smoothly so -- in addition they are not smoothly concordant to the positive Hopf link with a knot tied in the first component. Such examples were previously constructed by Cha--Kim--Ruberman--Strle; we show that our examples are distinct in smooth concordance from theirs. 
\end{abstract}

\maketitle


\section{Introduction}

The study of smooth and topological knot concordance can be considered to be a model for the significant differences between the smooth and topological categories in four dimensions. For instance, mirroring the fact that there exist 4--manifolds that are homeomorphic but not diffeomorphic, there exist knots that are topologically slice but not smoothly slice, i.e.\ knots that are topologically concordant to the unknot, but not smoothly so (see, for example, \cite{ChaPow14a, CHHo13, CHo15, End95, Gom86, HeK12, HeLivRub12, Hom14}). Similarly, one might ask whether there are links that are topologically concordant to the Hopf link, but not smoothly so. Infinitely many examples of such links were constructed by Cha--Kim--Ruberman--Strle in \cite{ChaKimRubStr12}. We construct another infinite family that we show to be distinct from the known examples in smooth concordance. 

In the following, all links will be considered to be ordered and oriented. Two links will be said to be concordant (resp.\ topologically concordant) if their (ordered, oriented) components cobound smooth (resp.\ topologically locally flat) properly embedded annuli in $S^3 \times [0,1]$. From now on, when we say the Hopf link, we refer to the positive Hopf link, i.e.\ the components are oriented so that the linking number is one.

\begin{figure}[t]
  \centering
  \includegraphics[width=4in]{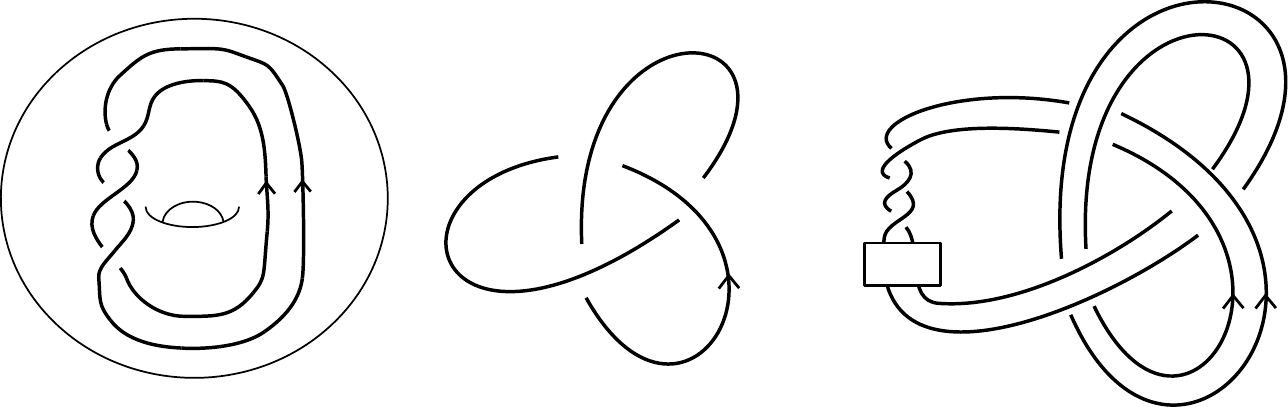}
  \put(-3.5,-0.20){$P$}
  \put(-2.25,-0.20){$K$}  
  \put(-0.85,-0.2){$P(K)$}  
	\put(-1.23,0.4){\small $3$}
  \caption{The (untwisted) satellite operation on knots. The boxes containing `$3$' indicate that all the strands passing vertically through the box should be given three full positive twists, to account for the writhe in the given diagram of $K$.}\label{fig:satellite}
\end{figure}
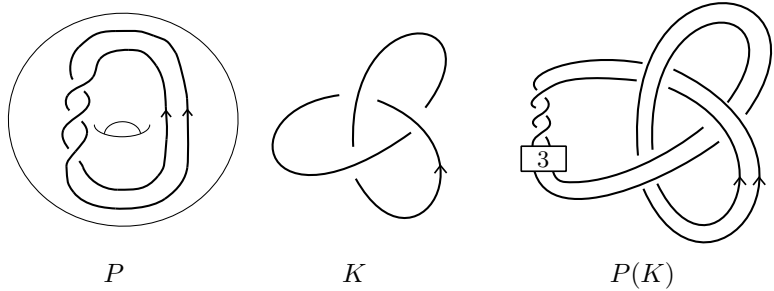

Any 2--component link with 
second component unknotted corresponds to a knot inside a solid torus, i.e.\ a \textit{pattern}, by carving out a regular neighborhood of the second component in $S^3$. Any pattern $P$ induces a function on the knot concordance group $\mathcal{C}$ via the usual satellite construction, called a \textit{satellite operator}, given by
\begin{align*}
P:\mathcal{C}&\rightarrow\mathcal{C}\\
K&\mapsto P(K)
\end{align*}
where $P(K)$ is the satellite knot with companion $K$ and pattern $P$; see Fig.~\ref{fig:satellite} and~\cite[p.\ 111]{Ro90} for more details. 

We will occasionally work in a slight generalization of usual concordance, which we now describe. We say that two knots are \textit{exotically concordant} if they cobound a smooth, properly embedded annulus in a smooth 4--manifold \textit{homeomorphic} to $S^3\times [0,1]$ (but not necessarily \textit{diffeomorphic}). $\mathcal{K}$ modulo exotic concordance forms an abelian group called the \textit{exotic knot concordance group}, denoted by $\mathcal{C}^\text{ex}$. If the 4--dimensional smooth Poincar\'e Conjecture is true, we can see that $\mathcal{C}=\mathcal{C}^\text{ex}$ \cite{CDR14}. Any pattern $P$ induces a well-defined satellite operator $P:\mathcal{C}^\text{ex}\rightarrow \mathcal{C}^\text{ex}$ mapping $K\mapsto P(K)$. 

Studying satellite operators can yield information about link concordance using the following proposition. 

\begin{proposition}[Proposition 2.3 of \cite{CDR14}\label{prop:diff_functions}]If the 2--component links $L_0$ and $L_1$ are concordant (or even exotically concordant), then the corresponding patterns $P_0$ and $P_1$ induce the same satellite operator on $\mathcal{C}^\text{ex}$, i.e.\ for any knot $K$, if the links $L_0$ and $L_1$ are smoothly concordant, then the knots $P_0(K)$ and $P_1(K)$ are exotically concordant. If $L_0$ and $L_1$ are topologically concordant, then $P_0(K)$ and $P_1(K)$ are topologically concordant. \end{proposition}

In the above formulation, the Hopf link corresponds to the identity function on $\mathcal{C}^\text{ex}$, and therefore, a 2--component link can be seen to be distinct from the Hopf link in smooth concordance if it induces a non-identity satellite operator on $\mathcal{C}^\text{ex}$. Similarly, the Hopf link with a knot $J$ tied into the first component corresponds to the connected-sum operator $C_J$ (where $C_J(K)=J\#K$ for all knots $K$), and therefore, a 2--component link can be seen to be distinct in smooth concordance from all links obtained from the Hopf link by tying a knot into the first component if the induced satellite operator is distinct from that of all connected-sum operators. 

This strategy can be considered to be a generalization of the method of distinguishing links by Dehn filling one component of the link or `blowing down one component', which can be seen to be the same as performing a twisted satellite operation using the unknot as companion (see, for example, \cite{ChaKo99, ChaKimRubStr12}, and item (7) of the highly useful list provided in \cite[Section 1]{FriedlPow14}).

For a winding number one pattern $P$ inside $ST=S^1\times D^2$ the standard unknotted solid torus, we let $\eta(P)$ denote the meridian $\{1\}\times \partial D^2$, oriented so that $\lk(P,\eta(P))=1$. Then it is easy to see that $P$ is the pattern corresponding to the link $(P,\eta(P))$. Moreover, patterns form a monoid (see~\cite[Section 2.1]{DR13}), and so we can consider the iterated patterns $P^i$, where $P^i(K)=P(P(\cdots(K)\cdots))$. $P^0$ is the core of an unknotted standard solid torus and we see that for any pattern $P$, the link $(P^0,\eta(P^0))$ is the Hopf link. 

\begin{figure}[b]
    \centering
		\includegraphics[width=3.5in]{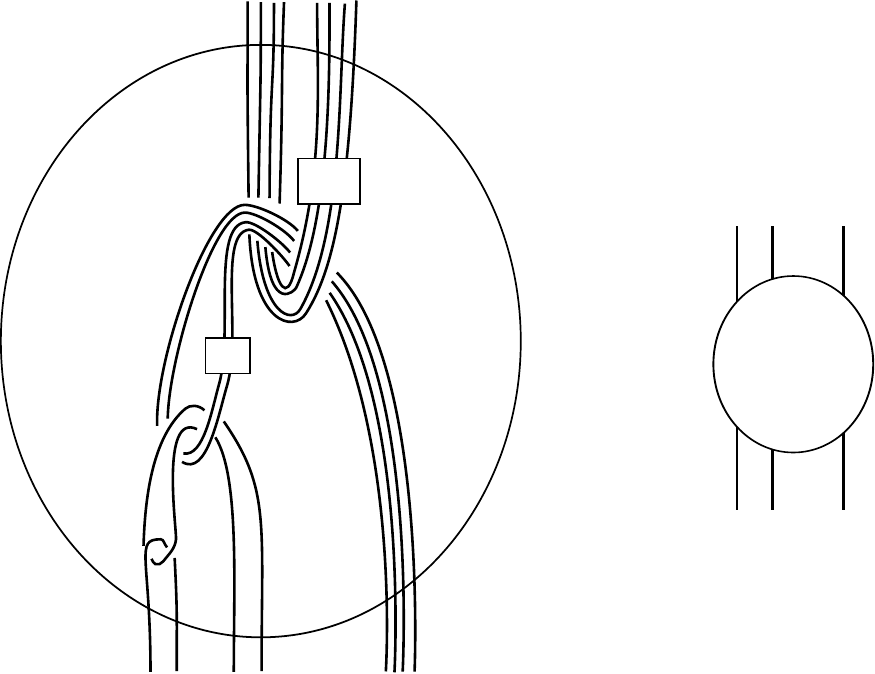}
    \put(-2.65,1.235){\small -2}
		\put(-2.315,1.925){ $-2$}
		\put(-0.55,1.15){\LARGE $\Wh_3$}
		\put(-1.2,1.15){\LARGE $=$}
		\put(-0.35,1.65){$\cdots$}
    \caption{Whenever we draw a circle containing a `$\Wh_3$', such as on the right, we mean the tangle shown on the left. The boxes containing `$-2$' indicate that all the strands passing vertically through the box should be given two full negative twists. As a result, the link in Fig.~\ref{fig:wh3link} is the link obtained by Whitehead doubling both components of the Whitehead link. }\label{fig:wh3}
\end{figure}
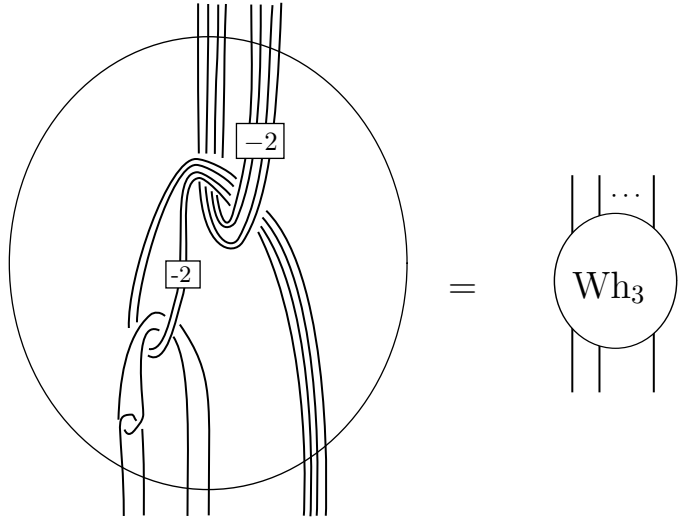	
							
Consider the link obtained by Whitehead doubling each component of the Whitehead link. By the symmetry of the link, we see that this is the same link as the one obtained by Whitehead doubling one component of the Hopf link three times (see Figs.~\ref{fig:wh3} and \ref{fig:wh3link}). For the rest of the paper, $L\equiv (Q,\eta)$ will refer to the link shown in Figure~\ref{fig:patternlink}, and $Q$ will denote the corresponding pattern or satellite operator.

\begin{figure}[t]
\centering
\includegraphics[width=2in]{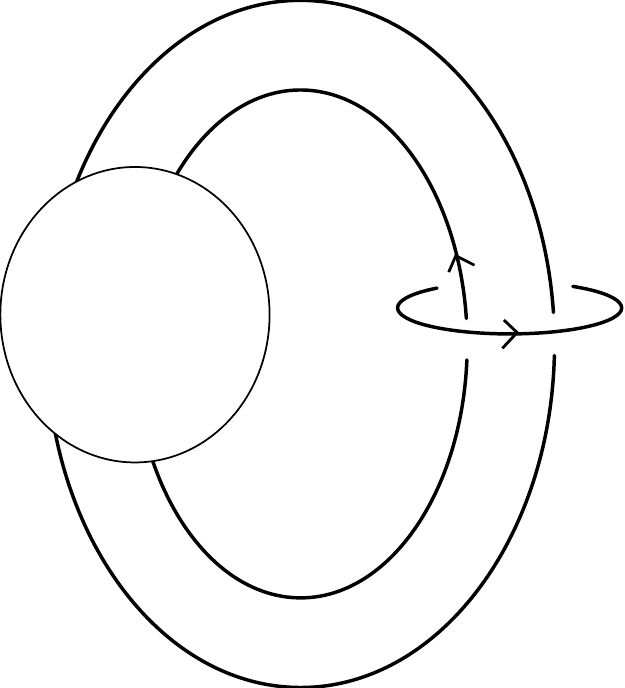}
                \put(-1.8,1.1){\LARGE $\Wh_3$}
		\put(-1.1, 2){\Large $\vdots$}
                \caption{$\Wh_3$, the link obtained by Whitehead doubling both components of the Whitehead link, or alternatively Whitehead doubling one of the components of the Hopf link three times. }\label{fig:wh3link}
\end{figure}

\begin{figure}[t]
		\centering
		\includegraphics[width=2in]{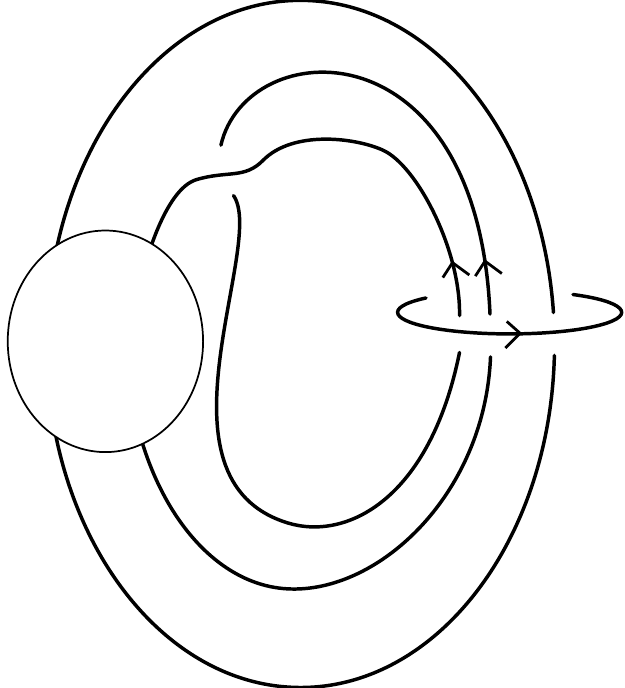}
    \put(-1.9,1.025){\LARGE $\Wh_3$}
		\put(-1.1, 2.03){\Large $\vdots$}
		\put(0, 1.1){$\eta$}
		\put(-1.85, 1.75){$Q$}
    \caption{The link $L\equiv (Q,\eta)$.}\label{fig:patternlink}
\end{figure}

\begin{theorem}\label{thm:main}The links $\{(Q^i,\eta(Q^i))\}$ are each topologically concordant to the Hopf link, but are distinct from the Hopf link (and one another) in smooth concordance. Moreover, they are distinct in smooth concordance from each link obtained from the Hopf link by tying a knot into the first component. For $i\geq 4$, they are distinct in smooth concordance from the Cha--Kim--Ruberman--Strle examples. \end{theorem}

\subsection*{Acknowledgments} While the authors were already in the initial stages of this project, the idea was independently suggested to the second author by an anonymous referee for~\cite{Ray15}. 
\section{Proofs}

The results of this section comprise Theorem~\ref{thm:main}.

\begin{proposition}\label{prop:topconc1}The 2--component link $(Q,\eta(Q))$ shown in Fig.~\ref{fig:patternlink} is topologically concordant to the Hopf link. \end{proposition}

\begin{figure}[t]
\centering
\includegraphics[width=2in]{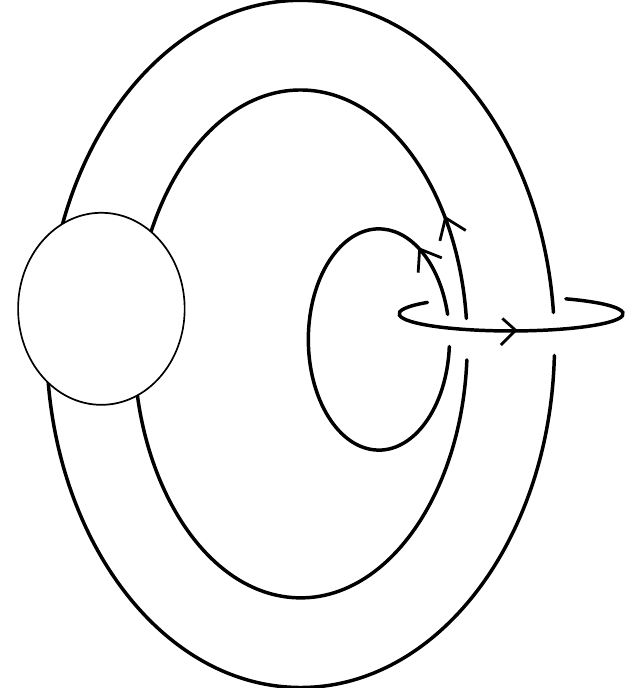}
\put(-1.9,1.125){\LARGE $\Wh_3$}
\put(-1.1, 2){\Large $\vdots$}
\put(-1.1,1.125){$a$}
\put(0,1.125){$b$}
\put(-1.875,0.5){$c$}
\caption{Proof of Theorem \ref{prop:topconc1}}\label{fig:proof1}
\end{figure}

\begin{proof} 
Let $L$ denote the link $(Q, \eta(Q))$.
By resolving a single crossing we get the link $(a, b, c)$ shown in Fig.~\ref{fig:wh3link}.  Thus, there is a cobordism from $(Q, \eta(Q))$ to $(a,b,c)$ consisting of an annulus $A$ bounded by $\eta(Q)\sqcup b$ and a pair of pants $P$ bounded by $Q\sqcup a \sqcup c$.  

Freedman proved that the link depicted in Fig.~\ref{fig:wh3link}, sometimes referred to as $\Wh_3$, is topologically slice in~\cite{Freed88}. We label the different components of Figure~\ref{fig:proof1} as $a$, $b$, and $c$ as shown. Note that $b\sqcup c$ is $\Wh_3$, the link shown in Fig.~\ref{fig:wh3link}, and $a\sqcup b$ is the Hopf link. Let $\Delta_b$, $\Delta_c\subseteq B^4$ be the disjoint slice disks for $b$ and $c$. By removing a regular neighborhood of a point on $\Delta_b$ we see that $c$ and its topological slice disk $\Delta_c$ in $S^3\times [0,1]$ is disjoint from a regular neighborhood of an annulus cobounded by $b$ in $S^3\times \{0\}$ and an unknot $U\subseteq S^3\times \{1\}$. Since $a$ is just a meridian of $b$, we can find an annulus entirely within this regular neighborhood, cobounded by $a\subset S^3\times \{0\}$ and a meridian of $U\subseteq S^3\times\{1\}$. 
Gluing these to $A\sqcup P$ gives a concordance between $L$ and the Hopf link.  
\end{proof}

\begin{remark}\label{rem:davis}
An alternative approach to Proposition~\ref{prop:topconc1} was suggested to the authors by Jim Davis, namely that if the multivariable Alexander polynomial of $L = (Q, \eta(Q))$ is one then $L$ is topologically concordant to the Hopf link by~\cite{Davis06}. We performed the computation using tools developed in ~\cite{Cooper79, CimFlo}, which we describe here. Given any link $L$ one can find a 2--complex $F$ called a \textit{C--complex} bounded by $L$ (see Figure~\ref{fig:C-complex}). Similar to the Seifert matrix one can generate a matrix by studying linking numbers between curves on $F$ and their pushoffs.  In \cite[Corollary 3.4]{CimFlo} and \cite[Chapter 2, Corollary 2.2]{Cooper79} it is shown that this matrix gives a presentation for the Alexander module of $L$. With respect to the C--complex $F$ and basis for $H_1(F)$ in Figure~\ref{fig:C-complex} that matrix is
$$
A = \left[\begin{array}{cccc}
0&1&0&0\\t_1&0&0&0\\0&0&0&1\\0&0&t_2&t_2-1\end{array}\right].
$$ Since  $\det(A) = t_1t_2$, the Alexander module of $L$ is trivial, and so by~\cite{Davis06} $L$ is topologically concordant to the Hopf link.  

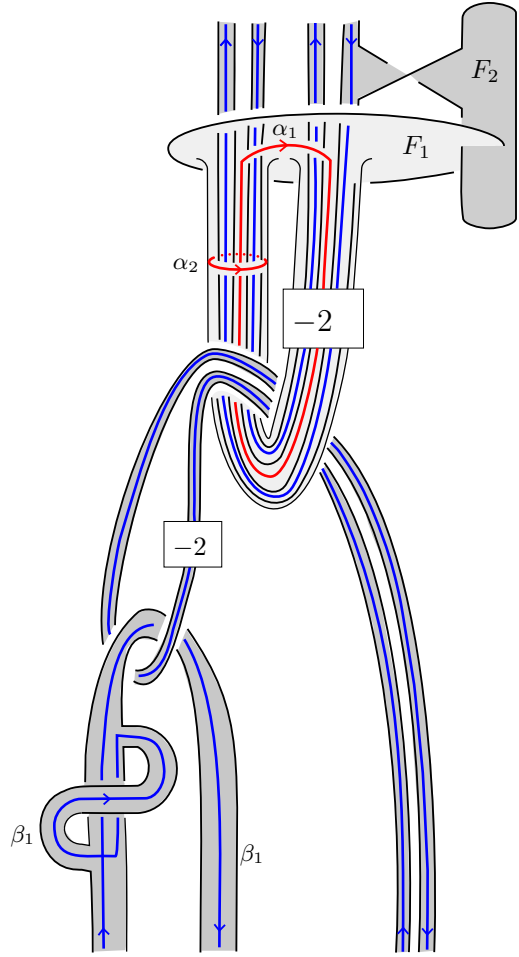
\begin{figure}[t]
\centering
\includegraphics[width=2.5in]{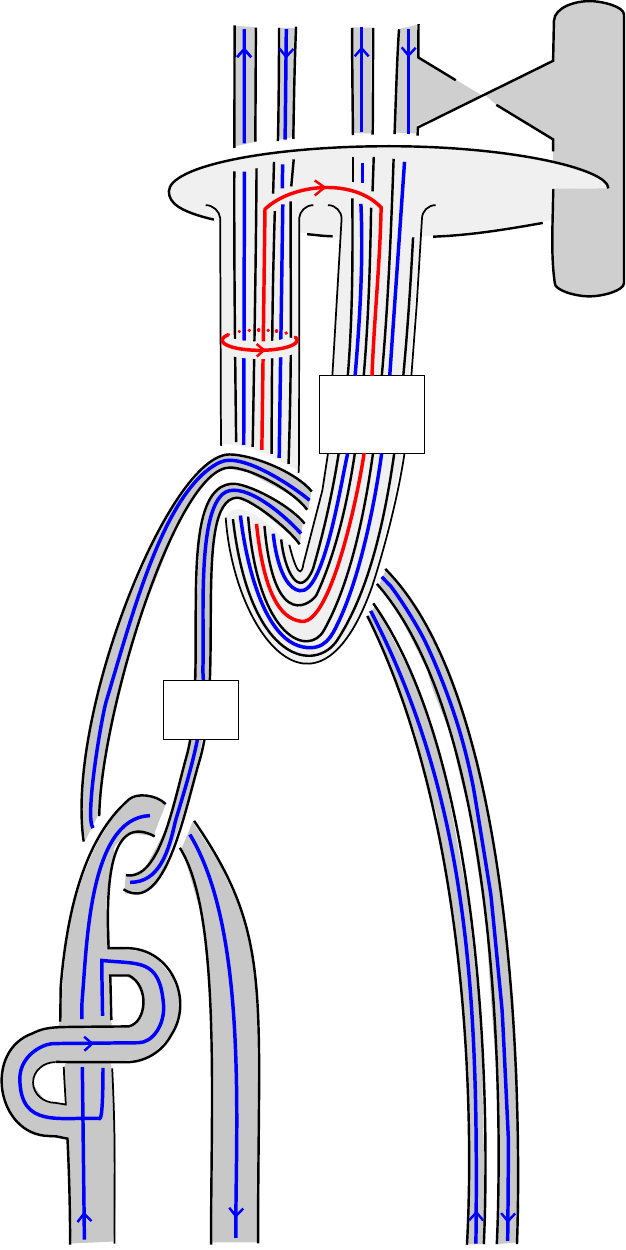}
    \put(-1.8,2.1){$-2$}
		\put(-1.175,3.275){\Large $-2$}
    \put(-.6,4.2){$F_1$}
    \put(-.24,4.6){$F_2$}
    \put(-1.28,4.3){\small{$\alpha_1$}}
    \put(-1.8,3.6){\small{$\alpha_2$}}
    \put(-1.45,.5){\small{$\beta_1$}}
    \put(-2.65,.6){\small{$\beta_1$}}
\caption{A C--complex $F=F_1\cup F_2$  for $L = (Q,\eta(Q))$; the four basis curves for $H_1$ are shown. }\label{fig:C-complex}
\end{figure}

This result, along with Theorem~\ref{thm:main}, shows that our links give another answer to the question posed by Jim Davis in~\cite[p.\ 266]{Davis06}, as did the examples of Cha--Kim--Ruberman--Strle. \end{remark}


\begin{proposition}\label{prop:topconc2}[See also Proposition 2.15 of \cite{DR13}] Each link of the form $(Q^i, \eta(Q^i))$, $i\geq 1$ is topologically concordant to the Hopf link.\end{proposition}

\begin{proof}This is essentially the proof that satellite operators are well-defined on concordance classes of knots. Since $(Q,\eta(Q))$ is topologically concordant to the Hopf link, we have two disjoint annuli $A_1$ and $A_2$ in $S^3\times [0,1]$, such that $A_1\cap S^3\times \{0\}=Q$, $A_2\cap S^3\times \{0\}=\eta(Q)$, and $(A_1 \sqcup A_2)\cap S^3\times \{1\}$ is the Hopf link. Cut out a regular neighborhood of $A_1$, and replace it with $ST\times[0,1]$, where $ST=S^3 - N(\eta)$ is a standard unknotted solid torus containing the pattern knot $Q$. The resulting manifold can be seen to be homeomorphic to $S^3\times [0,1]$. We obtain the link $(Q^2,\eta(Q^2))\subseteq S^3\times \{0\}$, the link $(Q,\eta(Q))\subseteq S^3\times \{1\}$, and a topological concordance between them in this new $S^3\times [0,1]$ given by $(Q\times [0,1])\sqcup A_2$. 

By iterating this process, we see that for each $i\geq 1$, the link $(Q^{i+1},\eta(Q^{i+1}))$ is topologically concordant to $(Q^i,\eta(Q^i))$. This completes the proof since, by Proposition \ref{prop:topconc1}, $(Q,\eta(Q))$ is topologically concordant to the Hopf link. \end{proof}

\begin{proposition}\label{prop:smoothlydistinct}The members of the family $\{(Q^i,\eta(Q^i))\mid i\geq 0\}$ are distinct from one another in smooth concordance. Moreover, for $i\geq 1$, they are each distinct in smooth concordance from any link obtained from the Hopf link by tying a knot in the first component. \end{proposition}

Recall that the link $(Q^0,\eta(Q^0))$ is the Hopf link, and therefore, the first statement above says that the links $(Q^i,\eta(Q^i))$ are distinct from the Hopf link in smooth concordance. 

\begin{proof}[Proof of Proposition~\ref{prop:smoothlydistinct}] For the first statement, consider the following proposition.

\begin{proposition}[\cite{Ray15}]\label{prop:distinctiterates}If $P$ is a winding number one pattern such that $P(U)$ is unknotted, where $U$ is the unknot, and $P$ has a Legendrian diagram $\mathcal{P}$ with $\tb(\mathcal{P}) > 0$ and $\tb(\mathcal{P}) + \rot(\mathcal{P}) \geq 2$, then the iterated patterns $P^i$ induce distinct functions on $\mathcal{C}^\text{ex}$, i.e.\ there exists a knot $K$ such that $P^i(K)$ is not exotically concordant to $P^j(K)$, for each pair of distinct $i,j\geq 0$. \end{proposition}

A Legendrian diagram $\mathcal{Q}$ for $Q$ with $\tb(\mathcal{Q})=2$ and $\rot(\mathcal{Q})=0$ is shown in Fig.~\ref{fig:legpattern}. It is clear that $Q(U)$ is unknotted. The first statement then follows from Proposition~\ref{prop:diff_functions}.

\begin{figure}[t]
\centering
\includegraphics[width=3in]{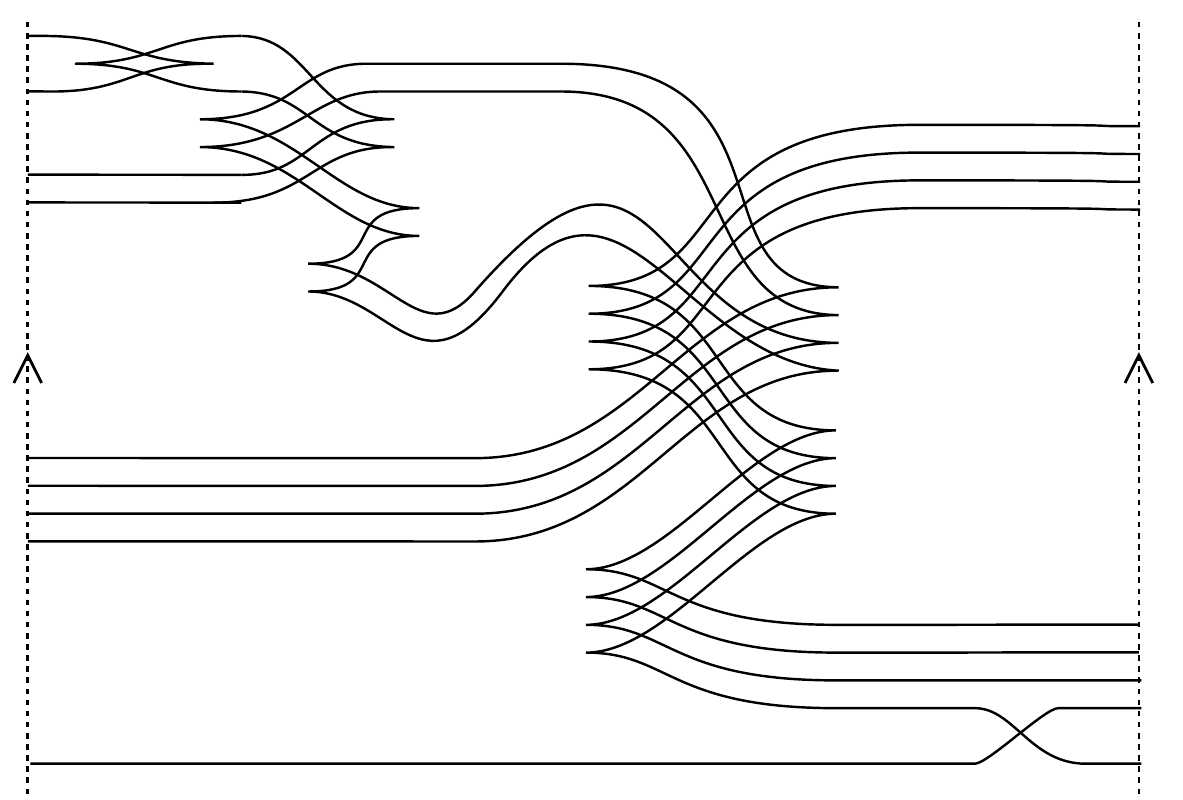}
\put(-2.25,-0.2){$\tb(\mathcal{Q})=2, \rot(\mathcal{Q})=0$}
\caption{A Legendrian diagram $\mathcal{Q}$ for the satellite operator $Q$. Note that this depicts a knot in a solid torus. }\label{fig:legpattern}
\end{figure}

If $(Q^i, \eta(Q^i))$ were concordant to a link obtained from the Hopf link by tying a knot $J$ into the first component, we know from Proposition~\ref{prop:diff_functions} that $Q^i(K)$ would be exotically concordant to $C_J(K)=J\#K$ for all knots $K$. By letting $K=U$ the unknot, since $Q^i(U)$ is unknotted, we see that $J$ must be exotically concordant to the unknot and as a result, $Q^i(K)$ is exotically concordant to $K$ for all knots $K$. But this contradicts Proposition~\ref{prop:distinctiterates} above, since $K=Q^0(K)$. \end{proof}

\begin{figure}[b]
\centering
\includegraphics[width=2in]{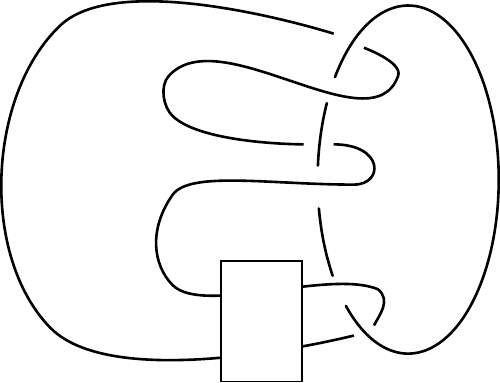}
\put(-1,0.2){$J$}
\caption{{The examples $\ell_J$ of Cha--Kim--Ruberman--Strle 
}}
\label{fig:oldexamples}
\end{figure}

The links constructed by Cha--Kim--Ruberman--Strle in \cite{ChaKimRubStr12} are of the form $\ell_J$ shown in Fig.~\ref{fig:oldexamples}. The box containing the letter $J$ indicates that all strands passing through the box should be tied into 0--framed parallels of a knot $J$. Cha--Kim--Ruberman--Strle showed that $\ell_J$ is topologically concordant to the Hopf link for all knots $J$, and that if $J$ is a knot with $\tau(J)>0$, $\ell_J$ is distinct from the Hopf link in smooth concordance. They also showed that if $J(n)=T(2,2n+1)$, the $(2,2n+1)$ torus knot, each member of the family $\{\ell_{J(n)}\}$ is smoothly distinct from the Hopf link (but topologically concordant to the Hopf link). 

\begin{figure}
\centering
\includegraphics[width=2in]{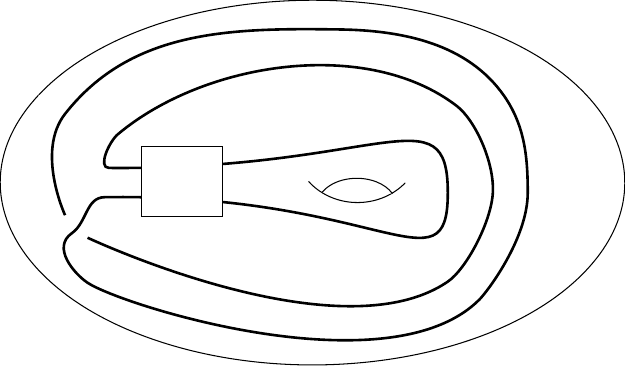}
\put(-1.475,0.55){$J$}
\caption{The patterns associated with the links $\ell_J$.}\label{fig:oldexamplespatterns}
\end{figure}

\begin{proposition}\label{prop:distinctfromCKRS} The links $\{(Q^i,\eta(Q^i))\mid i\geq 4\}$ are distinct in smooth concordance from the links $\ell_J$ constructed by Cha--Kim--Ruberman--Strle~\cite{ChaKimRubStr12}.\end{proposition}


\begin{proof}The Cha--Kim--Ruberman--Strle examples, as 2--component links with unknotted components, yield patterns as shown in Figure~\ref{fig:oldexamplespatterns}. Let $L_J$ denote the pattern knot obtained from the link $\ell_J$, and $L_J(K)$ denote the satellite knot obtained by applying $L_J$ to a knot $K$. From \cite[Theorem 1.2]{Rob12} we see that 
$$-n_+(L_J)-w \leq \tau(L_J(K)) - \tau(\widetilde{L_J}) - w\tau(K) \leq n_+(L_J) + w$$
 where $w$ is the winding number of $L_J$, $\widetilde{L_J} = L_J(U)$ is the result of erasing the second component of $\ell_J$, and $n_+(L_J)$ and $n_-(L_J)$ are the least number of positive and negative respectively intersections between $L_J$ and the meridian of the solid torus containing it. We see that $n_+(L_J)=2$, $\eta_-(L_J)=1$ and $w=1$. Since $\widetilde{L_J}$ is unknotted, $\tau(\widetilde{L_J})=0$. Therefore, if we let $K=RHT$ the right-handed trefoil, $$-2\leq \tau(L_J(RHT)) \leq 4,$$ since $\tau(RHT)=1$. Note that this does not depend on the choice of $J$. 

We will show that $\tau(Q^i(RHT)) >4$ for $i\geq 4$. By Proposition~\ref{prop:diff_functions}, this will complete the proof. Our main tool will be \cite[Theorem 1]{Plam04}, which states that if $\mathcal{K}$ is a Legendrian representative for a knot $K$, then 
$$\tb(\mathcal{K})+\lvert\rot(\mathcal{K})\rvert\leq 2\tau(K)-1.$$

We first build Legendrian representatives of the satellite knots $Q^i(RHT)$. We have a Legendrian diagram $\mathcal{Q}$ for the pattern $Q$ (Figure~\ref{fig:legpattern}.) We stabilize twice to get another Legendrian diagram $\mathcal{Q}'$ for $Q$ with $\tb(\mathcal{Q}')=0$ and $\rot(\mathcal{Q}')=2$. We can perform the Legendrian satellite operation on this Legendrian diagram by itself to get Legendrian diagrams $\mathcal{Q}'^i$ for the iterated patterns $Q^i$ since $\tb(\mathcal{Q}')=0$ (see \cite{Ng01} for background on the Legendrian satellite construction and \cite[Section 2.3]{Ray15} for details on this particular construction on patterns/satellite operators). By \cite[Lemma 2.4]{Ray15}, since $Q$ has winding number one, we see that $$\tb(\mathcal{Q}'^i)=i\cdot \tb(\mathcal{Q}')=0$$ and $$\rot(\mathcal{Q}'^i)=i\cdot\rot(\mathcal{Q}')=2i.$$

\begin{figure}
\centering
\includegraphics[width=1.5in]{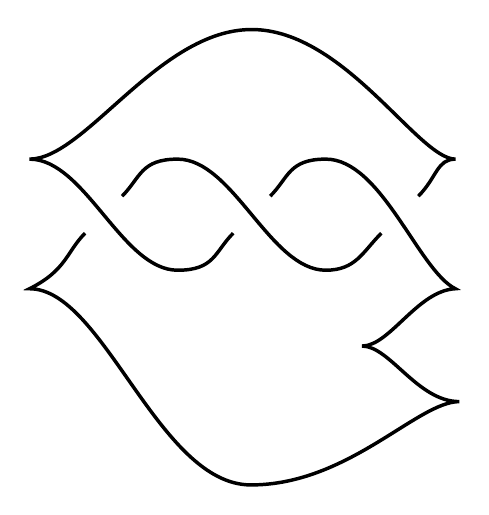}
\put(-1.55,-0.2){$\tb(\mathcal{K})=0$ and $\rot(\mathcal{K})=1$}
\caption{A Legendrian representative $\mathcal{K}$ for the right-handed trefoil.}\label{fig:legtrefoil}
\end{figure}

Consider the Legendrian representative $\mathcal{K}$ for the right-handed trefoil given in Fig.~\ref{fig:legtrefoil}. Since $\tb(\mathcal{K})=0$, we can perform the Legendrian satellite operation on $\mathcal{K}$ using the pattern $\mathcal{Q}'^i$ to get a Legendrian representative $\mathcal{Q}'^i(\mathcal{K})$ for the untwisted satellite $Q^i(RHT)$, and by \cite[Remark 2.4]{Ng01} we see that, since winding number of $Q^i$ is one,  
$$\tb(\mathcal{Q}'^i(\mathcal{K}))=\tb(\mathcal{Q}'^i)+\tb(\mathcal{K})=0$$
and $$\rot(\mathcal{Q}'^i(\mathcal{K}))=\rot(\mathcal{Q}'^i)+\rot(\mathcal{K})=2i+1.$$

Then using \cite[Theorem 1]{Plam04}, we see that 
$$\tb(\mathcal{Q}'^i(\mathcal{K}))+\lvert\rot(\mathcal{Q}'^i(\mathcal{K}))\rvert\leq 2\tau(Q^i(RHT))-1,$$
that is, 
$$i+1\leq \tau(Q^i(RHT)).$$
Therefore, if $i\geq 4$, $\tau(Q^i(RHT))>4$ as needed.  \end{proof}

\begin{remark} It is natural to ask whether our method would work for the links obtained by using $\Wh_1$ or $\Wh_2$ instead of $\Wh_3$, since those have many fewer crossings; these links are shown in Figure~\ref{fig:wh1and2}. The pattern corresponding to the link obtained by using $\Wh_1$ is called the \textit{Mazur pattern}, and has been widely studied, e.g.\ in~\cite{CFHeHo13, CDR14, Ray15, Lev14}. In~\cite{CFHeHo13} it was shown that the link using $\Wh_1$ is not topologically concordant to the Hopf link. For the link using $\Wh_2$, we can use a C--complex as in Remark~\ref{rem:davis} to compute the multivariable Alexander polynomial, which turns out to be 
$$-t_1^2t_2^2+2t_1^2t_2-t_1^2+2t_1t_2^2-3t_1t_2+2t_1-t_2^2+2t_2-1.$$ 
This can be used to show that this link is not topologically concordant to the Hopf link, using Kawauchi's result on the Alexander polynomials of concordant links in~\cite{Kaw78}.

\begin{figure}
\centering
\includegraphics{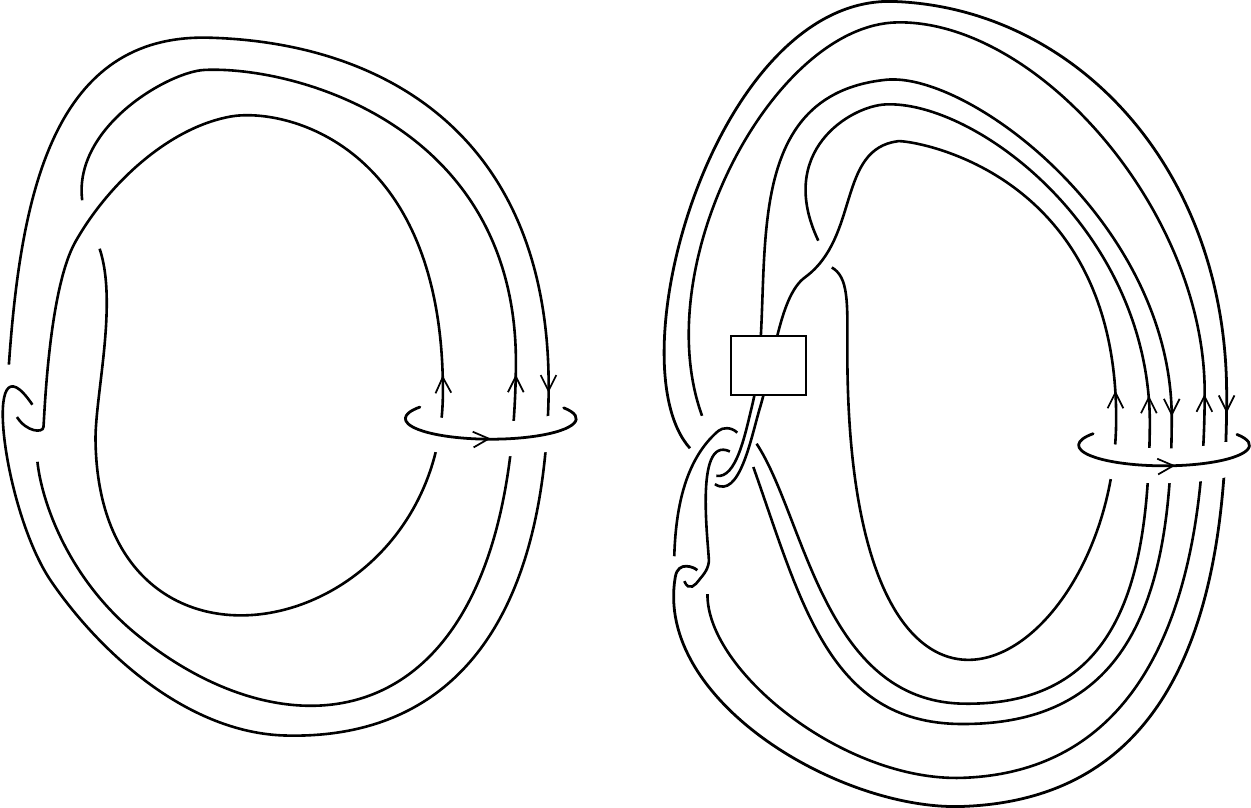}
\put(-2.025,1.725){$-2$}
\caption{The links obtained by using $\Wh_1$ (left) and $\Wh_2$ (right).}\label{fig:wh1and2}
\end{figure}

\end{remark}

Using our methods, we can prove the following theorem. 

\begin{theorem} Any 2--component link $(P,\eta)$ with linking number one, unknotted components, and Alexander polynomial one, where the corresponding pattern has a Legendrian diagram $\mathcal{P}$ with $\tb(\mathcal{P})>0$ and $\tb(\mathcal{P})+\rot(\mathcal{P})\geq 2$, yields a family of links $(P^i,\eta(P^i))$ that are each topologically concordant to the Hopf link, but are smoothly distinct from one another and the Hopf link.  \end{theorem}

For most values of $i$ (including $i\geq 4$, but possibly more values), the above links will be distinct in smooth concordance from the Cha--Kim--Ruberman--Strle examples, using the proof of Proposition~\ref{prop:distinctfromCKRS}. Using a more general version of Proposition~\ref{prop:distinctiterates} we can weaken our assumption that the first component of the link is unknotted, and instead require it to be slice and the pattern to be strong winding number one (see~\cite{Ray15}).

\bibliographystyle{plain}
\bibliography{knotbib}

\begin{thebibliography}{10}

\bibitem{ChaKimRubStr12}
Jae~Choon Cha, Taehee Kim, Daniel Ruberman, and Sa{\v{s}}o Strle.
\newblock Smooth concordance of links topologically concordant to the {H}opf
  link.
\newblock {\em Bull. Lond. Math. Soc.}, 44(3):443--450, 2012.

\bibitem{ChaKo99}
Jae~Choon Cha and Ki~Hyoung Ko.
\newblock On equivariant slice knots.
\newblock {\em Proc. Amer. Math. Soc.}, 127(7):2175--2182, 1999.

\bibitem{ChaPow14a}
Jae~Choon Cha and Mark Powell.
\newblock Covering link calculus and the bipolar filtration of topologically
  slice links.
\newblock {\em Geom. Topol.}, 18(3):1539--1579, 2014.

\bibitem{CimFlo}
David Cimasoni and Vincent Florens.
\newblock Generalized {S}eifert surfaces and signatures of colored links.
\newblock {\em Trans. Amer. Math. Soc.}, 360(3):1223--1264 (electronic), 2008.

\bibitem{CDR14}
Tim~D. Cochran, Christopher~W. Davis, and Arunima Ray.
\newblock Injectivity of satellite operators in knots concordance.
\newblock {\em J. Topol.}, 2014.
\newblock Advance Access published April 1, 2014.

\bibitem{CFHeHo13}
Tim~D. Cochran, Bridget~D. Franklin, Matthew Hedden, and Peter~D. Horn.
\newblock Knot concordance and homology cobordism.
\newblock {\em Proc. Amer. Math. Soc.}, 141(6):2193--2208, 2013.

\bibitem{CHHo13}
Tim~D. Cochran, Shelly Harvey, and Peter Horn.
\newblock Filtering smooth concordance classes of topologically slice knots.
\newblock {\em Geom. Topol.}, 17(4):2103--2162, 2013.

\bibitem{CHo15}
Tim~D. Cochran and Peter~D. Horn.
\newblock Structure in the bipolar filtration of topologically slice knots.
\newblock {\em Algebr. Geom. Topol.}, 4:415--428, 2015.

\bibitem{Cooper79}
D.~Cooper.
\newblock The universal abelian cover of a link.
\newblock In {\em Low-dimensional topology ({B}angor, 1979)}, volume~48 of {\em
  London Math. Soc. Lecture Note Ser.}, pages 51--66. Cambridge Univ. Press,
  Cambridge-New York, 1982.

\bibitem{DR13}
Christopher~W. Davis and Arunima Ray.
\newblock Satellite operators as group actions on knot concordance.
\newblock to appear: \textit{Alg.\ Geom.\ Topol.}, preprint:
  http://arxiv.org/abs/1306.4632, 2012.

\bibitem{Davis06}
James~F. Davis.
\newblock A two component link with {A}lexander polynomial one is concordant to
  the {H}opf link.
\newblock {\em Math. Proc. Cambridge Philos. Soc.}, 140(2):265--268, 2006.

\bibitem{End95}
Hisaaki Endo.
\newblock Linear independence of topologically slice knots in the smooth
  cobordism group.
\newblock {\em Topology Appl.}, 63(3):257--262, 1995.

\bibitem{Freed88}
Michael~H. Freedman.
\newblock {${\rm Whitehead}_3$} is a ``slice'' link.
\newblock {\em Invent. Math.}, 94(1):175--182, 1988.

\bibitem{FriedlPow14}
Stefan Friedl and Mark Powell.
\newblock Links not concordant to the {H}opf link.
\newblock {\em Math. Proc. Cambridge Philos. Soc.}, 156(3):425--459, 2014.

\bibitem{Gom86}
Robert~E. Gompf.
\newblock Smooth concordance of topologically slice knots.
\newblock {\em Topology}, 25(3):353--373, 1986.

\bibitem{HeK12}
Matthew Hedden and Paul Kirk.
\newblock Instantons, concordance, and {W}hitehead doubling.
\newblock {\em J. Differential Geom.}, 91(2):281--319, 2012.

\bibitem{HeLivRub12}
Matthew Hedden, Charles Livingston, and Daniel Ruberman.
\newblock Topologically slice knots with nontrivial {A}lexander polynomial.
\newblock {\em Adv. Math.}, 231(2):913--939, 2012.

\bibitem{Hom14}
Jennifer Hom.
\newblock The knot floer complex and the smooth concordance group.
\newblock {\em Comment.\ Math.\ Helv.}, 89(3):537--570, 2014.

\bibitem{Kaw78}
Akio Kawauchi.
\newblock On the {A}lexander polynomials of cobordant links.
\newblock {\em Osaka J. Math.}, 15(1):151--159, 1978.

\bibitem{Lev14}
Adam~Simon Levine.
\newblock Non-surjective satellite operators and piecewise-linear concordance.
\newblock Preprint, available at http://arxiv.org/abs/1405.1125, 2014.

\bibitem{Ng01}
Lenhard~L. Ng.
\newblock The {L}egendrian satellite construction.
\newblock Preprint: http://arxiv.org/abs/0112105, 2001.

\bibitem{Plam04}
Olga Plamenevskaya.
\newblock Bounds for the {T}hurston-{B}ennequin number from {F}loer homology.
\newblock {\em Algebr. Geom. Topol.}, 4:399--406, 2004.

\bibitem{Ray15}
Arunima Ray.
\newblock Satellite operators with distinct iterates in smooth concordance.
\newblock {\em Proc. Amer. Math. Soc.}, 143(11):5005--5020, 2015.

\bibitem{Rob12}
Lawrence~P. Roberts.
\newblock Some bounds for the knot {F}loer {$\tau$}-invariant of satellite
  knots.
\newblock {\em Algebr. Geom. Topol.}, 12(1):449--467, 2012.

\bibitem{Ro90}
Dale Rolfsen.
\newblock {\em Knots and links}, volume~7 of {\em Mathematics Lecture Series}.
\newblock Publish or Perish Inc., Houston, TX, 1990.
\newblock Corrected reprint of the 1976 original.

\end{thebibliography}
\end{document}